\documentclass[11pt]{amsart}

\usepackage{amsmath, amssymb, amsthm}
\usepackage{listings}
\usepackage{xifthen}
\usepackage{graphicx}
\usepackage{epstopdf}

\date{\today}

\theoremstyle{plain}
\newtheorem{theorem}{Theorem}[section]

\newtheorem{proposition}[theorem]{Proposition}

\newcommand{\niz}[4]{
\ifthenelse{\isempty{#1}}
{\def\a{f}}
{\def\a{#1}}
\ifthenelse{\isempty{#2}}
{\def\b{i}}
{\def\b{#2}}
\ifthenelse{\isempty{#3}}
{\def\c{1}}
{\def\c{#3}}
\ifthenelse{\isempty{#4}}
{\def\d{n}}
{\def\d{#4}}
\ensuremath{\left( {\a}_{\b} \right)_{\b = \c}^{\d}}
}
\newcommand{\nizf}{\niz{}{}{}{}}
\newcommand{\skup}[4]{
\ifthenelse{\isempty{#1}}
{\def\a{f}}
{\def\a{#1}}
\ifthenelse{\isempty{#2}}
{\def\b{i}}
{\def\b{#2}}
\ifthenelse{\isempty{#3}}
{\def\c{1}}
{\def\c{#3}}
\ifthenelse{\isempty{#4}}
{\def\d{n}}
{\def\d{#4}}
\ensuremath{\left\{ {\a}_{\b} \right\}_{\b = \c}^{\d}}
}
\newcommand{\skupf}{\skup{}{}{}{}}

\newcommand{\N}{\ensuremath{\mathbb{N}}}

\newcommand{\R}{\ensuremath{\mathbb{R}}}

\newcommand{\inner}[2]{\ensuremath{ {\left< #1 , #2 \right>} }}
\newcommand{\norm}[1]{\ensuremath{ {\left\| #1 \right\|} }}

\DeclareMathOperator{\ljuska}{span}

\DeclareMathOperator{\spn}{span}

\date{}
\title[Iterations of GGSP]{Iterations of the generalized Gram--Schmidt procedure for ge\-ne\-ra\-ting Parseval frames}
\author[T. Beri\' c]{Tomislav Beri\' c}

\address{Department of Mathematics, University of Zagreb,
Bijeni\v cka cesta 30, 10000 Zagreb, Croatia.}
\email{tberic@math.hr}

\begin{document}

\begin{abstract}
In this paper we describe some properties of the generalized Gram--Schmidt procedure (GGSP) for generating Parseval frames which was first introduced in~\cite{CasKut07}. Next we investigate the iterations of the procedure and its limit. In the end we give some examples of the iterated procedure.

\vspace{.1in}


\end{abstract}

\maketitle

\section{Introduction}
\label{intro}

Let $H$ be a finite--dimensional Hilbert space. A sequence $(f_i)_{i = 1}^{n}$ in $H$ is a \emph{frame} for $H$ if there exist constants $0 < A \le B < \infty$ such that
\begin{equation} \label{EqFrame}
A \norm{f}^2 \le \sum_{i = 1}^{n} \left| \inner{f}{f_i} \right|^2 \le B \norm{f}^2, \quad \text{for all } f \in H.
\end{equation}
Frames for Hilbert spaces were introduced in~\cite{DS} by R.J. Duffin and A.C. Schaeffer in 1952. In 1980's frames begun to play an important role in wavelet and Gabor analysis. Since then, frames are an important tool in both theoretical and applied mathematics. Frames have found a number of applications due to the inbuilt redundancy which provides resilience to noise and coefficient erasures. Among them, frames for which $A = B = 1$ in~\eqref{EqFrame}, called \emph{Parseval frames}, have proved to be most useful since they provide the same simple reconstruction formula as orthonormal bases, but with the added benefit of having redundancy. Explicitly, if $(f_i)_{i = 1}^{n}$ is a Parseval frame, then
$$
f = \sum_{i = 1}^{n} \inner{f}{f_i} f_i, \quad \text{for all } f \in H.
$$
The \emph{frame operator} $S : H \to H$ is defined as $S f = \sum_{i = 1}^{n} \inner{f}{f_i} f_i$. It is positive and invertible and from the equality $S^{-\frac{1}{2}} S S^{-\frac{1}{2}} = I$ we can easily get that $\left( S^{-\frac{1}{2}} f_i \right)_{i = 1}^{n}$ is a Parseval frame, a fact which we will use in the rest of the paper. For more details on frame theory we refer to the book~\cite{Christ03} or the survey article~\cite{Cas00}.

In \cite{CasKut07}, an algorithm was devised which generates Parseval frames using a generalization of the Gram--Schmidt orthogonalization procedure (or GGSP for short). For a given frame $(f_i)_{i = 1}^{n}$, the algorithm produces a Parseval frame $(g_i)_{i = 1}^{n}$ in the following manner: the first vector $g_1$ is simply the normalized vector $f_1$ as in the first step of the ordinary Gram--Schmidt algorithm. If $f_k \not \in \spn \{f_i\}_{i = 1}^{k-1}$, then $g_k$ is derived from the ordinary Gram--Schmidt step. If, on the other hand, $f_k$ is linearly dependent on the previous vectors, then $(g_i)_{i = 1}^{k}$ is the Parseval frame $S^{-1/2} \left( (g_i)_{i = 1}^{k-1} \cup \{f_k\} \right)$, where $S$ is the frame operator of the frame $(g_i)_{i = 1}^{k-1} \cup \{f_k\}$. In this step the previously generated vectors $g_1, \ldots, g_{k-1}$ have to be adjusted using the vector $f_k$. An important feature of this construction is that in each step $k$ we get a Parseval frame for $\spn \left( f_i \right)_{i = 1}^{k}$.

We will denote the mapping $(f_i)_{i = 1}^{n} \mapsto (g_i)_{i = 1}^{n}$ by $\Phi$. The algorithm's pseudocode is given below, verbatim as in~\cite{CasKut07}.

\begin{lstlisting}[language=Pascal, numbers=left, escapeinside='']
Procedure GGSP(n, f; g).
    for k := 1 to n do
    begin
        if '$f_k = 0$' then
            '$g_k := 0$';
        else
        begin
            '$g_k := f_k - \sum_{j = 1}^{k-1} \inner{f_k}{g_j} g_j$';
            if '$g_k \neq 0$' then
                '$g_k := \frac{1}{\norm{g_k}} g_k$';
            else
            begin
                for i := 1 to k - 1 do
                    '$g_i := g_i + \frac{1}{\norm{f_k}^2} \left( \frac{1}{\sqrt{1+\norm{f_k}^2}} - 1 \right) \inner{g_i}{f_k} f_k$';
                '$g_k := \frac{1}{\sqrt{1 + \norm{f_k}^2}} f_k$';
            end;
        end;
    end;
end.
\end{lstlisting}
We refer the reader to~\cite{CasKut07} for details on the algorithm and some of its properties.

The objective of this paper is to further investigate the algorithm with an emphasis on the iterations of the algorithm. Iterative algorithms are often employed in applications and in frame theory some notable examples of iterative procedures include the frame algorithm (\cite{Gro93}) and the gradient descent of the frame potential used for construction of approximate unit-norm tight frames (\cite{CasFicMix}). We are inspired here by these algorithms to study the limit case of the iterated GGSP.

We will end this section by introducing a term that will be important in the rest of the paper. We say that a sequence of vectors $\nizf$ in a Hilbert space $H$ is a \emph{zero extended orthonormal sequence} if the sequence becomes orthonormal once we remove all the zero vectors. We say that $\nizf$ is a \emph{zero extended orthonormal basis} if the reduced sequence is an orthonormal basis.

\section{Properties of the iterated GGSP}
\label{sec:1}

Since we will be dealing with the iterations of GGSP with the limit case in mind, the first thing we would like to know is which frames remain unchanged under the application of GGSP. It turns out that that set is the same as for the ordinary Gram--Schmidt procedure.

\begin{proposition} \label{PropStationaryGGSP}
Let $\nizf$ be a frame for a Hilbert space $H$. The following statements are equivalent:
\begin{itemize}
\item[(i)] $\Phi \left( \nizf \right) = \nizf$,

\item[(ii)] $\nizf$ is a zero extended orthonormal basis. 
\end{itemize}
\end{proposition}

\begin{proof}
If (ii) holds, GGSP becomes the ordinary Gram-Schmidt procedure (leaving zero vectors unchanged) so it doesn't change the orthonormal basis.

If we assume (i) is true, let $k \in \{1, 2, \ldots, n\}$ be the greatest index (if it exists) for which $f_k \neq 0$ and $f_k \in \ljuska \skup{}{}{}{k-1}$. The $k$--th vector can change only in the $k$--th step of the algorithm. So we must have $f_k = \dfrac{1}{\sqrt{1 + \norm{f_k}^2}} f_k$. Taking norms of both sides, we get $\norm{f_k}^2 = \dfrac{\norm{f_k}^2}{1 + \norm{f_k}^2}$ so $f_k$ has to be zero which is a contradiction.

Therefore, $\skup{f}{i}{1}{n} \setminus \{0\}$ is a linearly independent set. Therefore $k \le d$ ($d = \dim H$). Since GGSP for linearly independent sets is the regular Gram-Schmidt procedure, we have that $\Phi \left( \niz{f}{i}{1}{n} \right)$ is an orthonormal basis (possibly with zeros). Since we assumed $\Phi \left( \nizf \right) = \nizf$, it follows that $\nizf \setminus \{0\}$ is an orthonormal basis for $\ljuska \skupf = H$.
\end{proof}

Next we will turn our attention to the iterations of the GGSP and describe the limit case with regards to the $\ell^2$--norm. We denote by $G_0 = \left( g_{i}^{(0)} \right)_{i = 1}^{n} = \nizf$, the starting frame for $H$. Also, recursively we define the sequence
$$
G_{m+1} = \left( g_{i}^{(m+1)} \right)_{i = 1}^{n} := \Phi( G_m ), \quad m \ge 0,
$$
and
$$
G = \niz{g}{}{}{} := \left(\norm{\cdot}_2\right) \lim_{m \to \infty} \left( g_{i}^{(m)} \right)_{i = 1}^{n}
$$
if the limit exists. Notice that if the limit $\niz{g}{}{}{}$ exists, it is also a Parseval frame.

We will also adopt the following notation: since the input vectors may change more than once during the application of GGSP, we will have to observe not just the final vectors, but also the vectors as they appear in each step. Let us denote by $g_{i}^{(m, k)}$ the $i$--th vector we get in the $k$--th step of the $m$--th iteration of GGSP. We immediately see that $g_{i}^{(m, n)} = g_{i}^{(m)}$ for all $i$ and $m$.

There are two possibilities for the starting frame. Either the last vector is in the span of the preceding vectors, or it is not. The next proposition shows that we only need consider the first case when we study the convergence of iterations because in the latter case the last vector stabilizes right after the first iteration and has no effect on the other vectors.

\begin{proposition} \label{PropLastStabilizes}
If $f_{n} \notin \spn \{f_i\}_{i = 1}^{n-1}$, then for all $m > 1$ we have $g_{n}^{(m)} = g_{n}^{(1)} = \alpha (I - P) f_{n}$, where $P$ is the orthogonal projection onto $\spn \{f_i\}_{i = 1}^{n-1}$ and $\alpha = \norm{(I - P) f_{n}}^{-1}$.
\end{proposition}

\begin{proof}
Let us denote by $P$ the orthogonal projection onto $\spn \{f_i\}_{i = 1}^{n-1}$. After $n-1$ steps of the first iteration we get the vectors $g_{i}^{(1, n-1)} = g_{i}^{(1)}$, $i = 1, \ldots, n-1$, which form a Parseval frame for its span. Also, we have $\spn \{g_{i}^{(1)}\}_{i = 1}^{n-1} = \spn \{f_i\}_{i = 1}^{n-1}$. Now let's observe the $n$--th step:
\begin{align*}
g_{n}^{(1)} &= g_{n}^{(1, n)} = \alpha \left( f_n - \sum_{i = 1}^{n-1} \inner{f_n}{g_{i}^{(1)}} g_{i}^{(1)} \right) \\
&= \alpha \left( f_n - \sum_{i = 1}^{n-1} \inner{f_n}{P g_{i}^{(1)}} g_{i}^{(1)} \right) \\
&= \alpha \left( f_n - \sum_{i = 1}^{n-1} \inner{P f_n}{g_{i}^{(1)}} g_{i}^{(1)} \right) \\
&= \alpha \left( f_n - P f_n \right) = \alpha (I - P) f_n,
\end{align*}
where $\alpha = \norm{(I - P) f_{n}}^{-1}$ so that $g_{n}^{(1)}$ is a unit vector. Now, in the same way, in the second iteration we get
$$
g_{n}^{(2)} = (I - P) \alpha (I - P) f_n = \alpha (I - P)^2 f_n = \alpha (I - P) f_n = g_{n}^{(1)}
$$
which is again a unit vector so we don't need to normalize it. In each of the following iterations we will get the same vector.
\end{proof}

The case when the last vector is linearly dependent upon the previous vectors is actually just a special case of a more general result which we give in the next theorem.

\begin{theorem} \label{TmLinZavisniuNulu}
If $f_{k} \in \spn \{f_i\}_{i = 1}^{k-1}$, then $g_{k}^{(m)}$ converges to zero as $m$ tends to infinity.
\end{theorem}

\begin{proof}
Let us enumerate by $k_1 < k_2 < \ldots < k_s$ all the indices such that $f_{k_r} \in \ljuska \skup{}{}{}{k_r-1}$. For an arbitrary index $j$ among them, in the first iteration and the $j$--th step of GGSP we get the vector $g_{j}^{(1, j)}$. If $j < k_s$, the $j$--th vector will later change, let's say that the first time it happens is in the $k$--th step. The square of the new vector's norm will be:

\setlength{\abovedisplayskip}{1pt}
\setlength{\belowdisplayskip}{10pt}

\begin{align*}
\norm{ g_{j}^{(1, k)} }^2 &=
\norm{ g_{j}^{(1, j)} + \frac{1}{\norm{f_k}^2} \left( \frac{1}{\sqrt{1 + \norm{f_k}^2}} - 1 \right) \inner{g_{j}^{(1, j)}}{f_k} f_k }^2 =\\
&= \norm{g_{j}^{(1, j)}}^2 + 2 \frac{1}{\norm{f_k}^2} \left( \frac{1}{\sqrt{1 + \norm{f_k}^2}} - 1 \right) \left| \inner{g_{j}^{(1, j)}}{f_k} \right|^2 +\\
&+ \frac{1}{\norm{f_k}^4} \left( \frac{1}{\sqrt{1 + \norm{f_k}^2}} - 1 \right)^2 \left| \inner{g_{j}^{(1, j)}}{f_k} \right|^2 \cdot \norm{f_k}^2 = \\
&= \norm{g_{j}^{(1, j)}}^2 + \frac{1}{\norm{f_k}^2}  \left| \inner{g_{j}^{(1, j)}}{f_k} \right|^2 \left( \frac{1}{\sqrt{1 + \norm{f_k}^2}} - 1 \right) \left( 2 + \frac{1}{\sqrt{1 + \norm{f_k}^2}} - 1 \right) =\\
&= \norm{g_{j}^{(1, j)}}^2 - \frac{1}{1 + \norm{f_k}^2} \left| \inner{g_{j}^{(1, j)}}{f_k} \right|^2. 
\end{align*}
Since $j = k_l$ for some $l$, then in this manner the $j$--th vector will change in steps $k_{l+1}$, $k_{l+2}$, \ldots, $k_{s}$ giving us vectors $g_{k_l}^{(m, k_{l+1})}$, \ldots, $g_{k_l}^{(m, k_s)}$, respectively. By the previous calculation we see that for all $m \in \N$, $1 \le l \le s-1$, $l+1 \le r \le s$ we have

\setlength{\abovedisplayskip}{2pt}
\setlength{\belowdisplayskip}{0pt}

\begin{equation} \label{EqChange}
\norm{ g_{k_l}^{(m, k_r)} }^2 = \norm{ g_{k_l}^{(m, k_{r-1})} }^2 - \frac{1}{1 + \norm{ g_{k_r}^{(m - 1)} }^2} \left| \inner{ g_{k_l}^{(m, k_{r-1})} }{ g_{k_r}^{(m - 1)} } \right|^2.
\end{equation}
In particular, since $g_{k_l}^{(m, k_l)} = \frac{ g_{k_l}^{(m-1)} }{\sqrt{1 + \norm{ g_{k_l}^{(m-1)} }^2}}$, in the $k_{l+1}$--th step using the Cauchy-Schwarz inequality we get

\setlength{\abovedisplayskip}{2pt}
\setlength{\belowdisplayskip}{8pt}

\begin{align} \label{EqChange1}
\norm{ g_{k_l}^{(m, k_{l+1})} }^2 &= \norm{ g_{k_l}^{(m, k_l)} }^2 - \frac{1}{1 + \norm{ g_{k_{l+1}}^{(m - 1)} }^2} \left| \inner{ g_{k_l}^{(m, k_l)} }{ g_{k_{l+1}}^{(m - 1)} } \right|^2 \notag \\
&\ge \norm{ g_{k_l}^{(m, k_l)} }^2 - \frac{ \norm{ g_{k_l}^{(m, k_l)} }^2 \norm{ g_{k_{l+1}}^{(m - 1)} }^2 }{1 + \norm{ g_{k_{l+1}}^{(m - 1)} }^2} \notag \\
&= \norm{ g_{k_l}^{(m, k_l)} }^2 \left( 1 - \frac{ \norm{ g_{k_{l+1}}^{(m - 1)} }^2 }{1 + \norm{ g_{k_{l+1}}^{(m - 1)} }^2} \right) \notag \\
&= \frac{\norm{ g_{k_l}^{(m-1)} }^2}{1 + \norm{ g_{k_l}^{(m-1)} }^2} \cdot \frac{1}{1 + \norm{ g_{k_{l+1}}^{(m-1)} }^2}.
\end{align}
The $k_l$--th vector can change in this way a finite number of times ($s-l$ times, to be exact) and each time we get a decrease in norm as in~$\eqref{EqChange}$. In the end we will have a lower bound on the norm:
\begin{equation}  \label{EqChangeFinal}
\norm{ g_{k_l}^{(m)} }^2 \ge \frac{\norm{ g_{k_l}^{(m-1)} }^2}{1 + \norm{ g_{k_l}^{(m-1)} }^2} \cdot \frac{1}{1 + \norm{ g_{k_{l+1}}^{(m-1)} }^2} \cdots \frac{1}{1 + \norm{ g_{k_s}^{(m-1)} }^2}
\end{equation}

On the other hand, the vector $g_{k_s}^{(1)}$ satisfies:
$$
\norm{g_{k_s}^{(1)}}^2 = \frac{\norm{f_{k_s}}^2}{1 + \norm{f_{k_s}}^2}.
$$
We then have
$$
g_{k_s}^{(2)} = \frac{1}{\sqrt{1 + \norm{g_{k_s}^{(1)}}^2}} g_{k_s}^{(1)} = \frac{1}{\sqrt{1 + \dfrac{\norm{f_{k_s}}^2}{1 + \norm{f_{k_s}}^2}}} \cdot \frac{1}{\sqrt{1 + \norm{f_{k_s}}^2}} f_{k_s} = \frac{1}{\sqrt{1 + 2 \norm{f_{k_s}}^2}} f_{k_s}.
$$
Easy induction shows that
\begin{equation} \label{EqZadnjiuNulu}
g_{k_s}^{(m)} = \frac{1}{\sqrt{1 + m \norm{f_{k_s}}^2}} f_{k_s}, \quad m \in \N.
\end{equation}
Therefore, $\lim_{m \to \infty} \norm{g_{k_s}^{(m)}} = 0$. Now, for a fixed $0 < \varepsilon < 1$ we can find $m_0 \in \N$ such that $\norm{g_{k_s}^{(m)}}^2 < \varepsilon$, $\forall m \ge m_0$. For any $m > m_0$ from~\eqref{EqChange} we get
\begin{align*}
\norm{g_{k_{s-1}}^{(m)}}^2 &= \frac{ \norm{g_{k_{s-1}}^{(m-1)}}^2 }{ 1 + \norm{g_{k_{s-1}}^{(m-1)}}^2 } - \frac{1}{ 1 + \norm{g_{k_{s}}^{(m-1)}}^2 } \left| \inner{ \frac{ g_{k_{s-1}}^{(m-1)} }{\sqrt{1 + \norm{g_{k_{s}}^{(m-1)}}^2} } }{ g_{k_{s}}^{(m-1)} } \right|^2 = \\
&= \frac{ 1 }{ 1 + \norm{g_{k_{s-1}}^{(m-1)}}^2 } \left( \norm{g_{k_{s-1}}^{(m-1)}}^2 - \frac{ \left| \inner{ g_{k_{s-1}}^{(m-1)}  }{ g_{k_{s}}^{(m-1)} } \right|^2 }{ 1 + \norm{g_{k_{s}}^{(m-1)}}^2 } \right).
\end{align*}
We see that
\begin{equation} \label{IneqGornjaOgrada}
\norm{g_{k_{s-1}}^{(m)}}^2 \le \frac{ \norm{g_{k_{s-1}}^{(m-1)}}^2 }{ 1 + \norm{g_{k_{s-1}}^{(m-1)}}^2 }.
\end{equation}
Also, using~\eqref{EqChangeFinal} we get that
\begin{equation} \label{IneqDonjaOgrada}
\norm{g_{k_{s-1}}^{(m)}}^2 \ge \frac{ \norm{g_{k_{s-1}}^{(m-1)}}^2 }{ 1 + \norm{g_{k_{s-1}}^{(m-1)}}^2 } \frac{1}{1+\varepsilon}. \end{equation}
Combining~$\eqref{IneqGornjaOgrada}$~and~$\eqref{IneqDonjaOgrada}$ we see that
$$
\norm{g_{k_{s-1}}^{(m)}}^2 = \gamma_m \frac{ \norm{g_{k_{s-1}}^{(m-1)}}^2 }{ 1 +  \norm{g_{k_{s-1}}^{(m-1)}}^2 },
$$
for some $\gamma_m \in \left[ \frac{1}{1+\varepsilon}, 1 \right]$. Using the last result we can easily see that
\begin{align*}
\norm{g_{k_{s-1}}^{(m_0 + l)}}^2 &= \frac{ \gamma_{m_0 + 1} \cdots \gamma_{m_0 + l} \cdot \norm{g_{k_{s-1}}^{(m_0)}}^2 }{ 1 + \left( 1 + \gamma_{m_0 + 1} + \gamma_{m_0 + 1} \gamma_{m_0 + 2} + \ldots + \gamma_{m_0 + 1} \gamma_{m_0 + 2} \cdots \gamma_{m_0 + l - 1} \right) \norm{g_{k_{s-1}}^{(m_0)}}^2 } \le \\
&\le \frac{ \norm{g_{k_{s-1}}^{(m_0)}}^2 }{ \left( 1 + \frac{1}{1+\varepsilon} + \frac{1}{(1+\varepsilon)^2} + \ldots + \frac{1}{(1+\varepsilon)^{l-1}} \right) \norm{g_{k_{s-1}}^{(m_0)}}^2 } = \\
&= \frac{ 1 }{ \frac{ 1 - \frac{1}{(1+\varepsilon)^l}}{1 - \frac{1}{1+\varepsilon}} } = \frac{\frac{\varepsilon}{1 + \varepsilon}}{ \frac{ (1 + \varepsilon )^l - 1 }{ (1 + \varepsilon )^l } } = \frac{\varepsilon (1 + \varepsilon )^{l-1} }{ (1 + \varepsilon )^l - 1  } < 2 \varepsilon
\end{align*}
holds for big enough $l \in \N$.
Therefore, $\lim_{m \to \infty} \norm{g_{k_{s-1}}^{(m)}} = 0$ also. We would get that $\lim_{m \to \infty} \norm{g_{k_r}^{(m)}} = 0$, for all $r \in \{ 1, 2, \ldots, s \}$, analogously using the estimate~\eqref{EqChangeFinal} and the parameters $\gamma_m \in \left[ \frac{1}{(1+\varepsilon)^k}, 1 \right]$, for a suitable $k \in \N$.
\end{proof}

Observe that in no iterations will we get a zero vector if we didn't start with a zero vector, we can just get it in the limit case. Using Theorem~\ref{TmLinZavisniuNulu} we can now state our main result.

\begin{theorem} \label{TmMainResult}
The sequence $\left( G_m \right)_{m \ge 0}$ has a convergent subsequence in the $\ell^2$--norm for any choice of the starting sequence $\nizf{}{}{}{}$ and its limit $\niz{g}{}{}{}$ is a zero extended orthonormal basis for $\ljuska \skupf$. Moreover, $g_k = 0$ if and only if $f_k \in \ljuska \skup{f}{i}{1}{k-1}$.
\end{theorem}

\begin{proof}
First, let's observe that the square of the $\ell^2$--norm of any Parseval frame is equal to the dimension of the Hilbert space. Therefore, we have a sequence $\left( G_m \right)_{m \ge 0}$ of Parseval frames (except possibly for the starting sequence which can be an arbitrary frame), that is, it is a sequence of elements on the sphere of radius $\sqrt{d}$ in $H^n$, where $H$ is a $d$--dimensional Hilbert space. An easy compactness argument gives us a subsequence which converges to a limit $G$ which is also a Parseval frame. 

Now we have a limit Parseval frame $\niz{g}{}{}{}$ for a $d$ dimensional space which has exactly $n-d$ zero vectors. It is known that such a sequence must be an orthonormal basis with $n-d$ zeros added.
\end{proof}

The limit zero extended orthonormal basis in the previous theorem can be reached in some iteration only if we already started with a zero extended orthonormal basis. Otherwise, by Proposition~$\ref{PropStationaryGGSP}$, each iteration will yield a different Parseval frame.

Even though the sequence of Parseval frames produced by GGSP has a convergent subsequence for any starting frame, it is still unknown whether this sequence converges for any starting frame or, if it doesn't, a complete characterization of those frames that cause it to converge still remains an open problem.

\section{Examples}

In this section we will explore some numerical examples. In each of these examples of frames in $\R^2$ we can notice that after only a few iterations two vectors will stand out and form something that approximates an orthonormal basis. Other vectors will converge to zero. On the left we will have the starting frame together with the first $8$ iterations and on the right will be the $1000$-th iteration, which will serve to illustrate the limit Parseval frame.

First, we will choose the starting frame which has three vectors, out of which two are orthonormal (Figure~\ref{fig:1}).

\begin{figure}[ht]
\begin{center}
\includegraphics[width=0.4\textwidth]{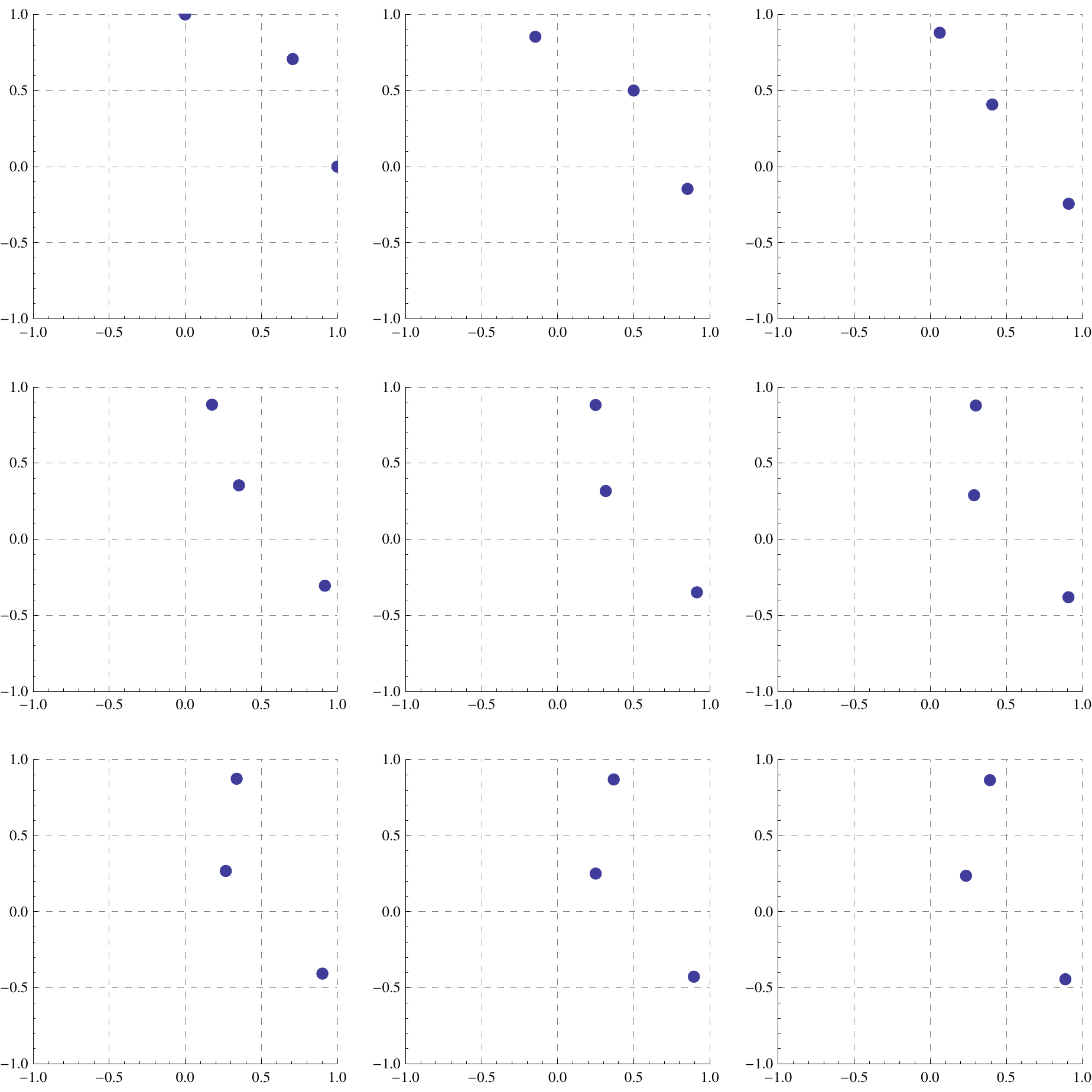}
\includegraphics[width=0.4\textwidth]{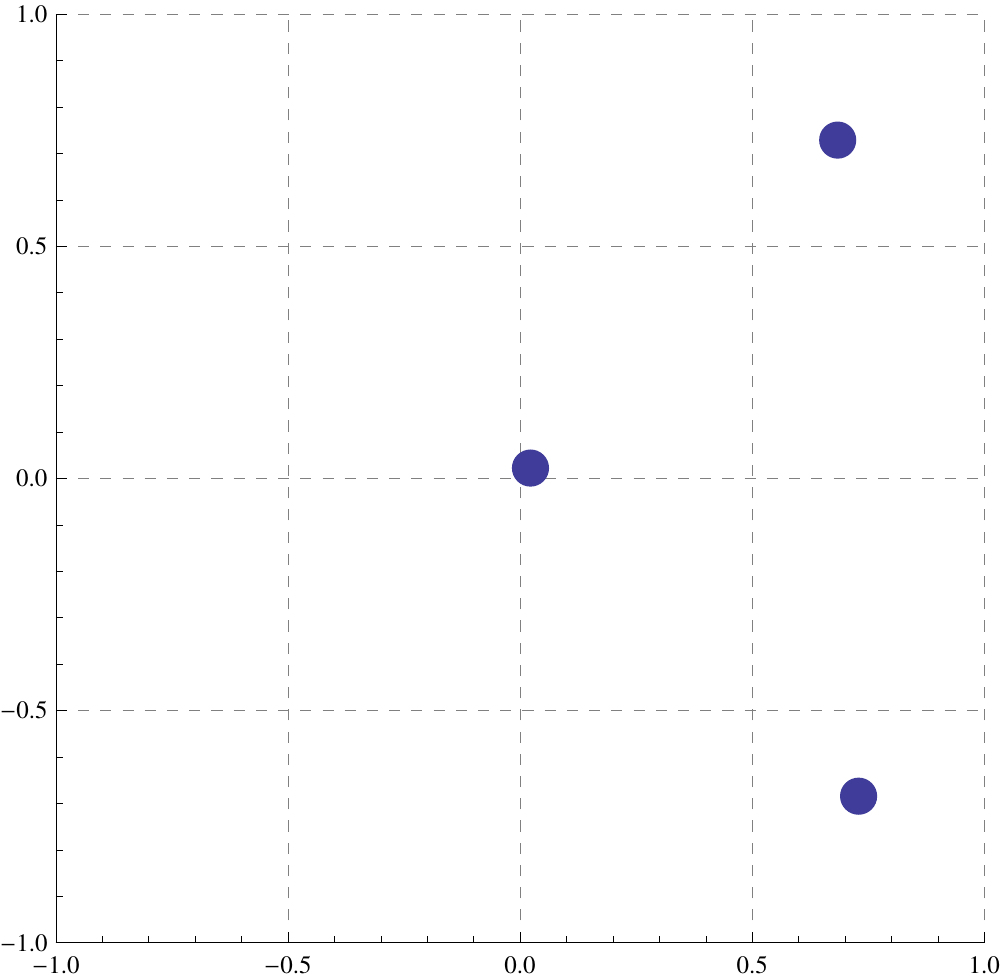}
\caption{On the left: the frame $\left\{ \left(1, 0\right), \left(0, 1\right), \left( 1/\sqrt{2},1/\sqrt{2} \right) \right\}$ together with the first $8$ iterates, on the right: its iteration limit}
\label{fig:1}
\end{center}
\end{figure}

Next we keep the two orthogonal vectors as before, but choose another third vector (Figure~\ref{fig:2}).

\begin{figure}[ht]
\begin{center}
\includegraphics[width=0.4\textwidth]{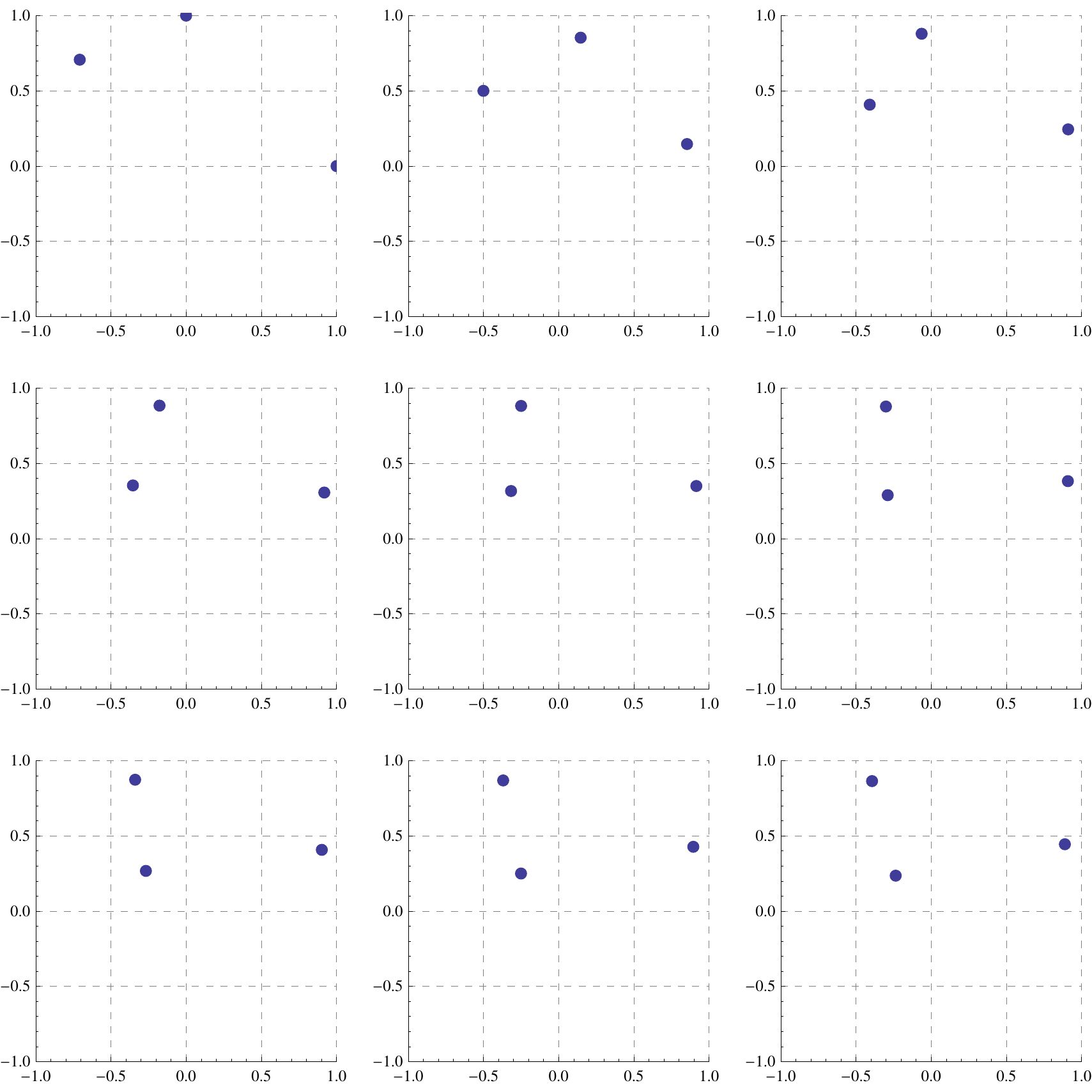}
\includegraphics[width=0.4\textwidth]{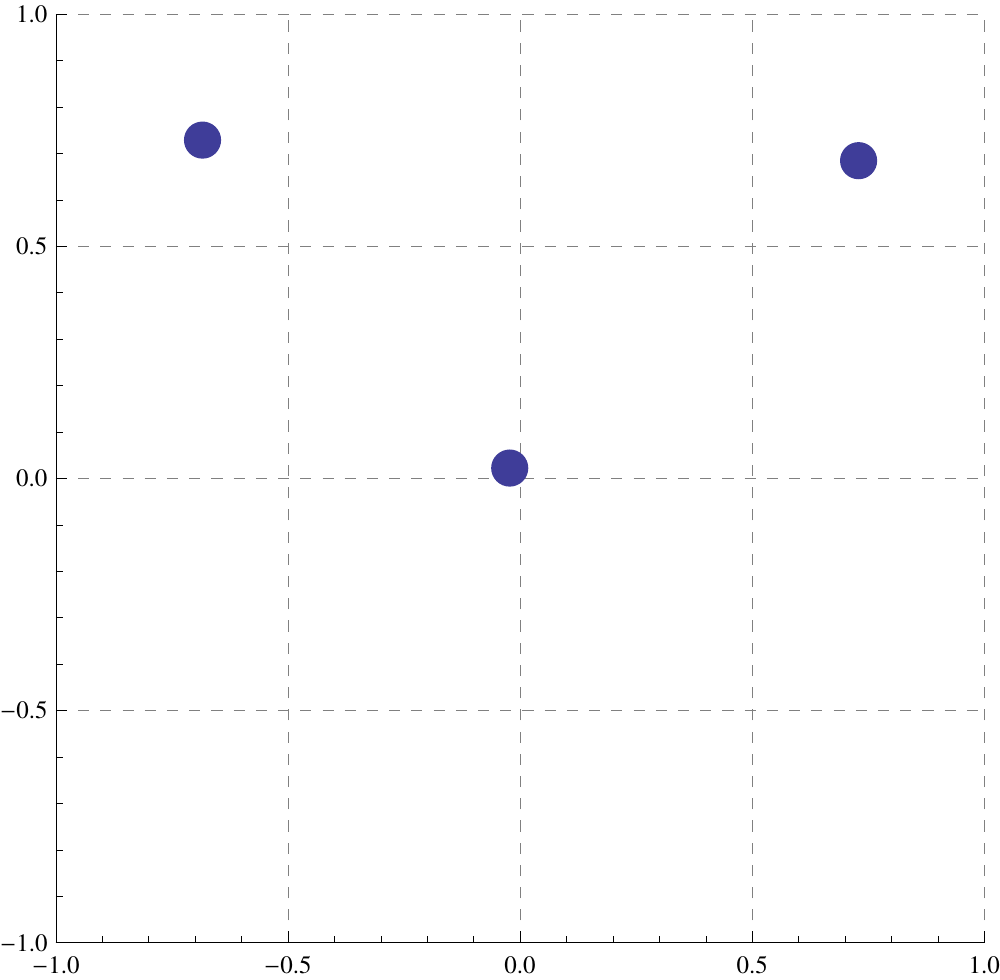}
\caption{On the left: the frame $\left\{ \left(1, 0\right), \left(0, 1\right), \left( -1/\sqrt{2},1/\sqrt{2} \right) \right\}$ together with the first $8$ iterates, on the right: its iteration limit}
\label{fig:2}
\end{center}
\end{figure}

We finish with a nice example in which the geometry preserving property of the algorithm is apparent in each iteration, but after a couple of iterations two vectors start to stand out which will form an orthonormal basis in the limit (Figure~\ref{fig:3}).
\pagebreak

\begin{figure}[ht]
\begin{center}
\includegraphics[width=0.4\textwidth]{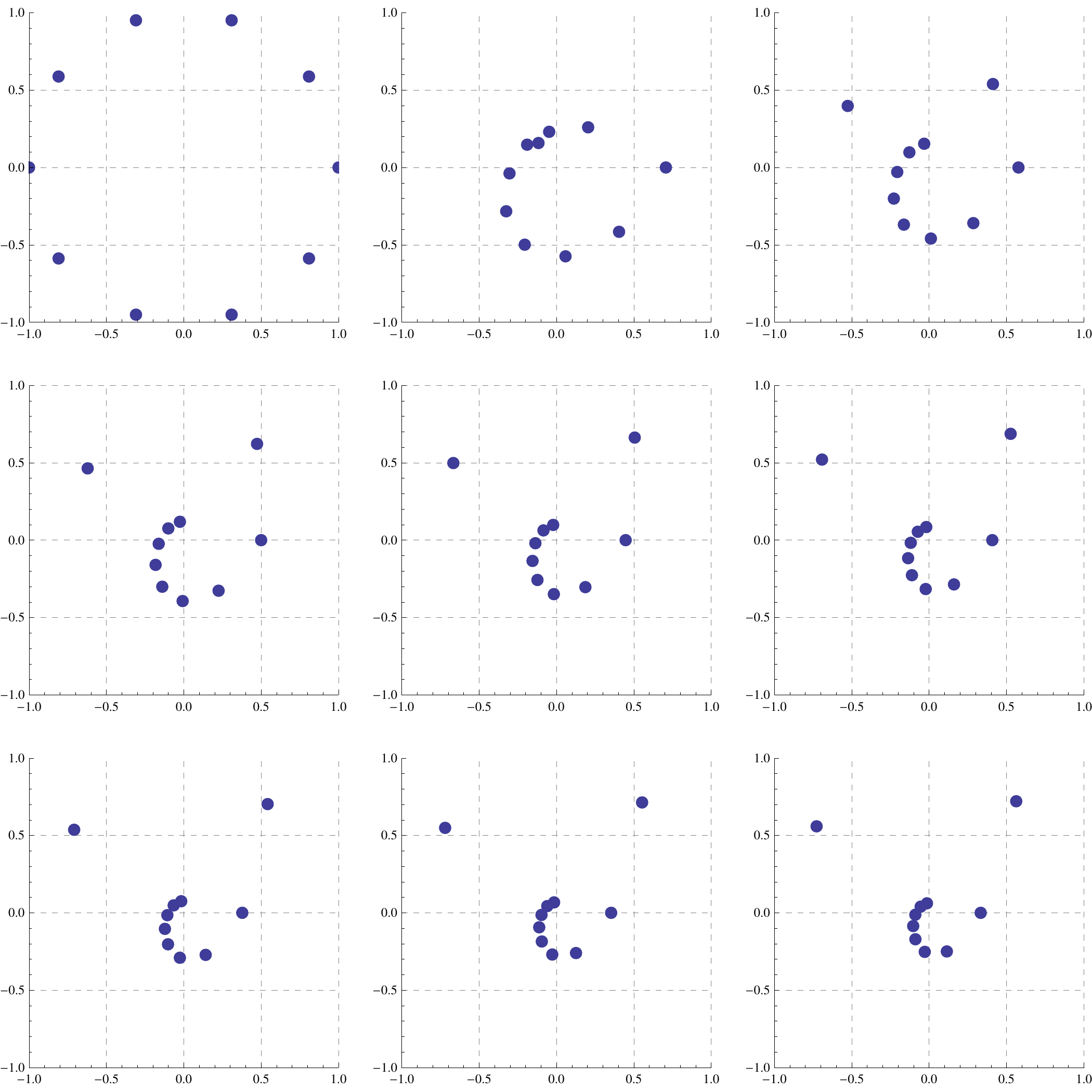}
\includegraphics[width=0.4\textwidth]{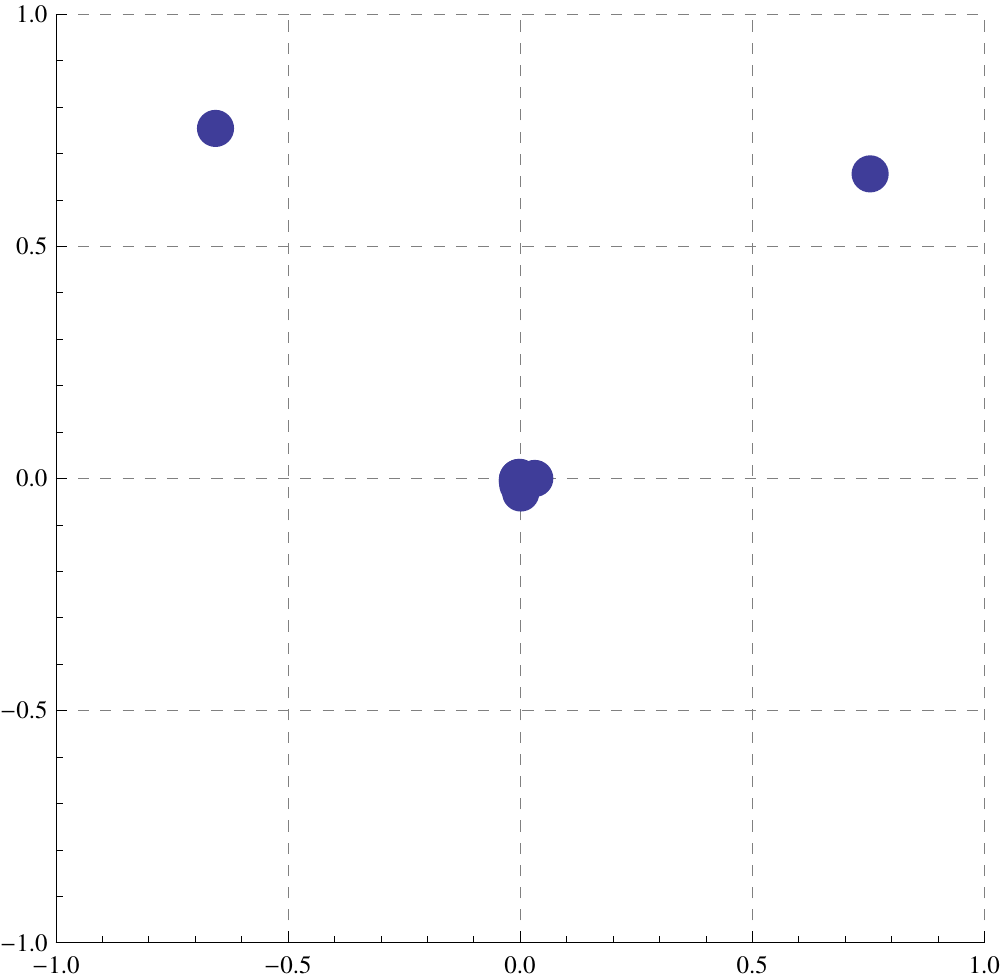}
\caption{On the left: the frame $\{ (\cos (2 k\pi/10), \sin (2 k\pi/10)) \}_{k = 1}^{10}$ together with the first $8$ iterates, on the right: its iteration limit}
\label{fig:3}
\end{center}
\end{figure}

\section*{Acknowledgements}

The author would like to thank Liljana Aramba\v{s}i\'{c} and Damir Baki\'{c} for many interesting and stimulating talks on this subject as well as for providing comments on how to improve on some results.

\end{document}